\documentclass{amsart}

\usepackage[margin=1.4in, marginpar=.7in, marginparsep=.1in]{geometry}

\usepackage[]{inputenc}

\allowdisplaybreaks

\usepackage{hyperref}
\usepackage{cite}
\usepackage{amssymb}
\usepackage{amsmath}
\usepackage{amsthm}
\usepackage{amsfonts}
\usepackage{graphicx}
\usepackage{bbm}
\usepackage[toc,page]{appendix}
\usepackage{subfigure}
\usepackage{mathtools}
\usepackage{bm} 
\usepackage[normalem]{ulem}
\usepackage{comment}
\usepackage{epsfig}
\usepackage{floatrow}

\newtheorem{theorem}{Theorem}[section]

\newtheorem{proposition}[theorem]{Proposition}
\newtheorem{corollary}[theorem]{Corollary}

\newtheorem{lemma}[theorem]{Lemma}
\newtheorem*{theorem*}{Theorem}

\theoremstyle{definition} 
\newtheorem*{remark}{Remark}
\newtheorem{definition}[theorem]{Definition}

\newtheorem*{acknowledgement}{Acknowledgment}

\numberwithin{equation}{section}

\newcommand{\abs}[1]{\left\lvert #1 \right \rvert}

\newcommand{\m}[1]{\mathbb{#1}}

\newcommand{\R}{\mathbb{R}}  
\newcommand{\Prob}{\mathbb{P}}

\newcommand{\hC}{\widehat{\mathbb{C}}}

\def\g{\gamma}

\def\ii{\mathfrak{i}}

\def\SLE{\operatorname{SLE}}

\def\BL{\operatorname{BL}}
\def\Chat{\hat{\m{C}}}

 \newcommand{\splus}{{\scriptstyle +}}

\def\dd{\mathrm{d}}

\usepackage{graphicx} 

\usepackage[dvipsnames]{xcolor}

\begin{document}

\title[Onsager--Machlup functional 
for SLE loop measures]{Onsager--Machlup functional 
for $\text{SLE}_{\kappa}$ loop measures}

\author[Carfagnini]{Marco Carfagnini{$^{\dag }$}}
\address{ Department of Mathematics 
University of California, San Diego \\
La Jolla, CA 92093-0112,  U.S.A.}
\email{mcarfagnini@ucsd.edu}

\author[Wang]{Yilin Wang{$^{\dag }$}}
\address{Institut des Hautes Études Scientifiques (IHES)\\ 
Bures-sur-Yvette, 91440, France}
\email{yilin@ihes.fr}

\date{September 2023}

\begin{abstract}
We relate two ways to renormalize the Brownian loop measure on the Riemann sphere. One by considering the Brownian loop measure on the sphere minus a small disk, known as the normalized Brownian loop measure; the other by taking the measure on simple loops induced by the outer boundary of the Brownian loops, known as Werner's measure. This result allows us to interpret the Loewner energy as an Onsager--Machlup functional for SLE$_\kappa$ loop measure for any fixed $\kappa \in (0, 4]$, and more generally, for any Malliavin--Kontsevich--Suhov loop measure of the same central charge.
\end{abstract}

\maketitle

\section{Introduction}

Onsager--Machlup functionals were introduced in \cite{MachlupOnsager1953a, MachlupOnsager1953b} to determine the most probable path of a diffusion process and can be considered as a probabilistic analog of the Lagrangian of a dynamical system.  For instance, let $(B_{t})_{0\leqslant t \leqslant  1}$ be a standard real-valued Brownian motion and $\phi$ be a smooth real-valued curve starting at the origin. Then 
\begin{equation}\label{eqn.OM.functional}
    \lim_{\varepsilon \rightarrow 0} \frac{\Prob\left( \max_{0\leqslant t\leqslant 1} \vert B_{t} - \phi_{t} \vert <\varepsilon \right)}{\Prob \left( \max_{0\leqslant t\leqslant 1} \vert B_{t}  \vert <\varepsilon \right)} = \exp \left( -\mathcal{O} (\phi) \right),
\end{equation}
where $\mathcal{O}(\phi)$  is the Onsager--Machlup functional with respect to the sup-norm. For a Brownian motion  $\mathcal{O}(\phi)= \frac{1}{2}\int_{0}^{1} \phi^{\prime} (t)^{2} dt$ coincides with the Dirichlet energy of the curve $\phi$. 
One can consider \eqref{eqn.OM.functional} for different classes of stochastic processes, or smoothness of the curve $\phi$, and finally for tubes around the trajectories defined by different norms, see, \cite{Zeitouni1989, LyonsZeitouni1999,SheppZeitouni1993, ChaoDuan2019,carfagninigordina2023}. 
A simple change of variable gives for $\kappa > 0$,
\begin{equation}\label{eqn.OM.kappa.B}
    \lim_{\varepsilon \rightarrow 0} \frac{\Prob\left( \max_{0\leqslant t\leqslant 1} \vert \sqrt \kappa B_{t} - \phi_{t} \vert <\varepsilon \right)}{\Prob \left( \max_{0 \leqslant t\leqslant 1} \vert \sqrt \kappa B_{t}  \vert <\varepsilon \right)} = \exp \left( -  \frac{\mathcal{O} (\phi)}{\kappa} \right).
\end{equation}

Let $\nu^\kappa$ denote the law of $\sqrt \kappa B$ on the space $C_{0}([0,1], \R)$ of real-valued continuous functions starting from zero. Then \eqref{eqn.OM.kappa.B} can be restated as 
\begin{equation}\label{eqn.OM.functional2}
    \lim_{\varepsilon \rightarrow 0} \frac{\nu^\kappa \left( D_{\varepsilon} (\phi) \right)}{\nu^\kappa \left( D_{\varepsilon} (0) \right)} = \exp \left( -\frac{\mathcal{O} (\phi)}{\kappa} \right),
\end{equation}
where $D_{\varepsilon} (\phi)$ denotes the ball of radius $\varepsilon$ in $C_{0}([0,1], \R)$ with respect to the sup-norm centered at $\phi \in C_{0}([0,1], \R)$.

This work aims to identify the Onsager--Machlup functional for SLE$_\kappa$ loop measures. 
 SLE loop measure is a one-parameter family (indexed by $0<\kappa \le 4$) of infinite, $\sigma$-finite measures on simple loops in the Riemann sphere $\hC = \mathbb C \cup \{\infty\}$, which is moreover invariant under conformal automorphisms of $\hC$ (for $ 4 < \kappa < 8$ the SLE$_\kappa$ loop measure is supported on non-simple loops). It arises from scaling limits of critical lattice models and is constructed in \cite{Zhan2017} by Zhan as a natural loop analog of the chordal SLE$_\kappa$ curve connecting two boundary points of a simply connected domain by Schramm \cite{schramm2000scaling}. See also an earlier construction of the SLE$_{8/3}$ loop measure \cite{Werner2008} and that of the SLE$_2$ loop measure \cite{Benoist_loop}. The SLE loop measure exhibits more symmetries as loops are considered unparametrized and do not have any distinguished marked point (as opposed to the boundary points of chordal SLE). 

\bigskip

To state our main theorem, let us first identify the set of neighborhoods in the space of simple loops that we consider. 
Let $\gamma$ be an analytic loop such that $\gamma = f (S^{1})$ for some conformal map $f$ defined on $\mathbb{A}_{r} = \{z \in \mathbb C \,|\, r< |z| < 1/r\}$ for some $ 0 < r <1 $. For $0 < \varepsilon < 1-r$,  set $A_{\varepsilon}:= \mathbb A_{1-\varepsilon}$
and let us consider the neighborhoods of $S^{1}$ and $\gamma$ given by 
\begin{align}
    & O_{\varepsilon} (S^{1} ) := \left\{ \text{non-contractible simple loops in } A_{\varepsilon} \right\}, \notag
    \\
    & O_{\varepsilon} (\gamma ) := \left\{ \text{non-contractible simple loops in } f(A_{\varepsilon}) \right\}.    \label{eqn.admissible.neigh}
\end{align}
We call the sets of simple loops of the form $O_{\varepsilon} (\gamma )$ as \emph{admissible neighborhoods}. We show that such neighborhoods are not scarce. In fact, every open set for the topology on the space of simple loops induced by the Hausdorff metric on $\hC$ is a union of such neighborhoods (Proposition~\ref{prop.basis.topology}). However, we cautiously note that these neighborhoods $O_{\varepsilon} (\gamma )$ are \emph{not} open sets for the Hausdorff metric as they do not contain the contractible loops that are Hausdorff-close to $\g$.

The following is our 
main theorem. 

\begin{theorem}\label{thm.main}
Let $\kappa \le 4$ and $\mu^{\kappa}$ be the $\SLE_{\kappa}$ loop measure. For any analytic simple loop $\gamma$ and a collection of admissible neighborhoods $(O_{\varepsilon} (\gamma))_{0 <\varepsilon \ll 1}$ defined as above, we have that 
\begin{align}\label{eqn.main}
\lim_{\varepsilon \rightarrow 0} \frac{\mu^{\kappa}  \left( O_{\varepsilon} (\gamma )\right) }{\mu^{\kappa}  \left( O_{\varepsilon} (S^{1} )\right) } = \exp \left( \frac{c(\kappa)}{24}I^{L}(\gamma) \right),
\end{align}
where $c(\kappa):= (6-\kappa)(3\kappa-8)/2\kappa$ is the central charge of $\SLE_{\kappa}$ and $I^L(\gamma)$ is the Loewner energy of $\gamma$.
\end{theorem}

In other words, the functional $\frac{c(\kappa)}{24} I^{L}$ can then be viewed as an Onsager--Machlup functional for the $\text{SLE}_{\kappa}$ loop measure on the space of simple loops. 
\begin{remark}
    We note that $\SLE_\kappa$ loop has the Loewner driving function $\sqrt \kappa B$ where $B$ is a two-sided standard Brownian motion on $\R$, and $I^L(\gamma)$ is defined as the Dirichlet energy of the driving function of $\gamma$ as introduced in \cite{RohdeWang2021,W2}.  We will recall the definition and basic properties of the Loewner energy in Section~\ref{sec:Loewner}.
    
    In light of \eqref{eqn.OM.functional2}, a natural guess for the asymptotics \eqref{eqn.main} would be
    \begin{align*}
\lim_{\varepsilon \rightarrow 0} \frac{\mu^{\kappa}  \left( O_{\varepsilon} (\gamma )\right) }{\mu^{\kappa}  \left( O_{\varepsilon} (S^{1} )\right) } = \exp \left( -\frac{I^{L}(\gamma)}{\kappa} \right).
\end{align*}
Theorem~\ref{thm.main} shows this guess is only true in the large deviation ($\kappa \to 0\splus$) regime since $ c(\kappa)\sim - 24 / \kappa$. In fact, all previous relations between the Loewner energy and SLE are stated in the $\kappa \to 0\splus$ or $\kappa \to \infty$ regime. This work is the first time we relate the Loewner energy to SLE$_\kappa$ with a fixed $\kappa$. See \cite{Wang_survey} for a survey on large deviation principles of SLE.
We also note that for $8/3 < \kappa \le 4$, $c(\kappa) \in (0,1]$ has a different sign than $-1/\kappa$. 
Theorem~\ref{thm.main} can be guessed from the expression of the Loewner energy in terms of determinants of Laplacians as shown in \cite{W2} and the partition function of SLE \cite{Dubedat2009partition}.
    The Loewner energy also appears in various other contexts, see, e.g., \cite{TT06,johansson2021strong}. 
\end{remark}

\begin{remark}
     The proof of Theorem~\ref{thm.main} only uses the fact that SLE$_\kappa$ loop measure is a Malliavin--Kontsevich--Suhov loop measure (introduced in \cite{Kontsevich_SLE} and proved in \cite{Zhan2017} for SLE loops) with central charge $c(\kappa)$. It is not yet known if there is a unique (up to scaling) Malliavin--Kontsevich--Suhov loop measure for a given central charge. \emph{A priori}, our result applies more generally to any Malliavin--Kontsevich--Suhov loop measure.
\end{remark}

   Malliavin--Kontsevich--Suhov loop measure is characterized by a property known as the conformal restriction covariance, stated using the normalized Brownian loop measure $\Lambda^{\ast}$ introduced in \cite{LaurenceLawler2013}. 
    Let  $V_{1}$ and $V_{2}$ be two disjoint compact non-polar subsets of $\hC$, $\Lambda^{\ast} (V_1, V_2) \in \mathbb R$ is defined as a renormalization of total mass of loops intersecting both $V_1$ and $V_2$ under Brownian loop measure\footnote{To relate to \cite{Kontsevich_SLE}, the total mass of Brownian loop measure can be interpreted as $-\log \det_\zeta \Delta$ where $\Delta$ is the Laplace-Beltrami operator as pointed out in \cite{LeJan2006det,Dubedat2009partition}.}. We recall that a set $V$ is \emph{non-polar} if, for every $z\in \mathbb C$, there is a positive probability that a Brownian motion started at $z$ hits $V$.   See Section~\ref{sec:BLM} for more details.

\bigskip

  On the other hand, the Loewner energy for simple loops 
  satisfies a similar conformal restriction property stated with Werner's measure $\mathcal W$ introduced in \cite{Werner2008}. 
  Werner's measure is an infinite measure on \emph{simple} loops and is obtained as the outer boundary of the Brownian loops in $\mathbb C$ and coincides with a multiple of the SLE$_{8/3}$ loop measure.  
  We write $\mathcal W (V_1, V_2)$ as the total mass of loops under Werner's measure intersecting both $V_1$ and $V_2$.

A key step of the proof of Theorem~\ref{thm.main} is to relate $\Lambda^{\ast} (V_1, V_2)$ to $\mathcal W (V_1, V_2)$ and then use the conformal restriction property for both SLE loop measure and Loewner energy.

\begin{theorem}[See Lemma~\ref{lemma.werner.not.lambda} and Corollary~\ref{cor.lambda.is.Werner}]\label{thm.intro.lambda.werner}
 In general, $\Lambda^{\ast} (V_1, V_2)$ and $\mathcal W (V_1, V_2)$ are not equal.
 For instance, let $K$ be a non-polar compact set and $\{K_{n}\}_{n\in \mathbb{N}}$ be a decreasing family of compact non-polar subsets shrinking to a point and disjoint from $K$. Then 
\begin{align*}
    & \lim_{n\rightarrow \infty} \Lambda^{\ast} \left( K, K_{n} \right) = - \infty,  & \lim_{n\rightarrow \infty} \mathcal{W} \left( K, K_{n} \right) = 0. 
\end{align*}
        However, let $A$ be an annulus, $\gamma \subset A$ be a simple loop, and $f: A\rightarrow f(A)$ be a conformal map. Then we have that
    \begin{align}
       &  \mathcal{W} \left( \gamma, A^{c} \right) - \mathcal{W} ( f(\gamma), f(A)^{c} )    =\Lambda^{\ast} \left( \gamma, A^{c} \right) - \Lambda^{\ast} ( f(\gamma), f(A)^{c} ). 
    \end{align} 
\end{theorem}

See also Theorem~\ref{thm.lambda.werner.compact}. 

\section{Brownian loop measure and Werner's measure} \label{sec:loop_measures}

\subsection{Brownian loop measure} \label{sec:BLM}
In this section, we recall the definition and main features of the Brownian loop measure on Brownian-type non-simple loops and Werner's measure on SLE$_{8/3}$-type simple loops. 

The Brownian loop measure $\mu^{\BL}$ on the Riemann sphere $\hC$  was introduced by Lawler and Werner in \cite{LawlerWerner2004} using an integral of weighted two-dimensional Brownian bridges, and its definition can be easily generalized to any Riemannian surface. See, e.g., \cite[Sec.\,2]{W3}. For an open set $D \subset \hC$,  we write $\mu^{\BL}_D$ for the measure $\mu^{\BL}$ restricted to the loops contained in $D$. The Brownian loop measure satisfies the following properties:

-\textit{Restriction invariance}: if $D^{\prime}\subset D$, then $ \dd \mu_{D^{\prime}}^{\BL} (\cdot) = 1_{\cdot \subset D^{\prime} } \,\dd \mu_{D}^{\BL}(\cdot)$.

-\textit{Conformal invariance}: if $D$ and $D^{\prime}$ are conformally equivalent domains in the plane, the pushforward of $\mu_{D}^{\BL}$ via any conformal map from $D$ to $D^{\prime}$, is exactly  $\mu_{D^{\prime}}^{\BL}$.

The mass of Brownian loops passing through one point is zero. Therefore, we also consider $\mu^{\BL}$ as the Brownian loop measure on $\mathbb C$.
The total mass 
of loops contained in $\mathbb{C}$ is infinite from both the contribution of big and small loops. 
To get rid of small loops, we consider $V_{1}$ and $V_{2}$ two compact disjoint non-polar subsets of $\mathbb{C}$, and let $\mathcal{L}\left(V_{1}, V_{2} \right)$ be the family of loops in $\mathbb{C}$ that intersect both $V_{1}$ and $V_{2}$. 
But $\mu^{\BL} \left( \mathcal{L}\left(V_{1}, V_{2} \right) \right)$ is still infinite due to the presence of big loops. 
In contrast, if $D$ is a proper open subset of $\mathbb{C}$ with non-polar boundary containing $V_{1}$ and $V_{2}$, then $\mu^{\BL}_{D} \left( \mathcal{L}\left(V_{1}, V_{2} \right) \right)$ is finite (since big loops in $\mathbb C$ will eventually hit  $\partial D$ and are not included in $\mu^{\BL}_{D}$). We write 
$$\mathcal B (V_1, V_2; D) : = \mu^{\BL}_{D} \left( \mathcal{L}\left(V_{1}, V_{2} \right) \right).$$

The \emph{normalized Brownian loop measure} $\Lambda^{\ast}$  introduced in \cite{LaurenceLawler2013} is defined by 
\begin{equation}\label{eqn.lambda.ast}
\Lambda^{\ast} \left( V_{1}, V_{2} \right) := \lim_{r\rightarrow 0} \left( \mathcal{B} \left( V_{1}, V_{2} ; \, \mathbb{D}_{r} (z_{0})^c   \right)    - \log \log \frac{1}{r}   \right),
\end{equation}
where $z_{0} \in \mathbb{C}$ and $\mathbb{D}_{r} (z_{0})^c  = \hC \setminus \mathbb{D}_{r} (z_{0}) = \{z \in \hC \, |\, |z-z_0| > r\}$. 
It was proved in \cite{LaurenceLawler2013} that the limit \eqref{eqn.lambda.ast} converges to a finite number if $V_{1}$ and $V_{2}$ are disjoint non-polar compact subsets of $\hC$, and the value does not depend on the choice of $z_{0}$. 
One can also choose $z_0 = \infty$, then
\begin{equation}\label{eqn.lambda.ast.infty}
\Lambda^{\ast} \left( V_{1}, V_{2} \right) = \lim_{R\rightarrow \infty} \left( \mathcal{B} \left( V_{1}, V_{2} ; \, \mathbb{D}_{R}   \right)  - \log \log R   \right),
\end{equation}
where $\mathbb D_R := \mathbb D_R(0)$.

We remark that as Lemma~\ref{lemma.werner.not.lambda} shows, 
$\Lambda^{\ast}$ is not induced from a measure in the sense that $\Lambda^{\ast} (V_1, V_2)$ cannot be written as the total mass of loops intersecting $V_1$ and $V_2$ under some positive measure. Nonetheless, $\Lambda^{\ast}(\cdot, \cdot)$ satisfies M\"obius invariance, that is, 
$$\Lambda^{\ast} \left( A(V_{1}) , A(V_{2}) \right) = \Lambda^{\ast} \left( V_{1}, V_{2} \right)$$ for any M\"obius transformation $A$ of the Riemann sphere. Moreover, we have the following relation which follows directly from the definition.
\begin{lemma}\label{lem:lambda_*_restriction}
   Let $D^\prime \subset D$ be proper subdomains of $\mathbb C$ and $K \subset D^\prime$ is a non-polar compact set. Then 
   $$\Lambda^{\ast} (K, D^{\prime c}) = \Lambda^{\ast} (K, D^{c}) + \mathcal B (K, D\setminus D^\prime; D).$$ 
\end{lemma}

The Brownian loop measure has been used in describing the conformal restriction covariance property of chordal SLE \cite{LSW_CR_chordal} while the ambient domain is the upper half-plane, hence renormalization is not needed. While we consider variants of SLE on the Riemann sphere, such as the case of whole-plane SLE in \cite{LaurenceLawler2013}, the right normalization applied to Brownian loop measure so that similar conformal restriction formula holds is given by $\Lambda^{\ast}$.

\subsection{SLE loop measures and conformal restriction covariance}

In \cite{Zhan2017}, Zhan constructed the SLE$_\kappa$ loop measure and showed its \emph{conformal restriction covariance} with central charge $c (\kappa) = (6-\kappa)(3\kappa-8)/2\kappa \le 1$ for $\kappa \le 4$. More precisely, this means a consistent family of measures $(\mu^\kappa_D)_D$ on 
\emph{simple} loops contained in a subdomain $D \subset \hC$ which satisfies:

-\textit{Restriction covariance}: for $D \subset \hC$, then
\begin{equation}\label{eqn.conformal.restr}
\frac{d \mu_{D}^\kappa}{ d \mu^\kappa} \left( \cdot \right) = 1_{ \left\{ \cdot \, \subset D \right\} } J
_\kappa \left( \cdot, D^{c} \right),
\end{equation}
where $\mu^\kappa = \mu^\kappa_{\hC}$ and $J_\kappa \left( \cdot, D^{c} \right):= \exp \left( \frac{c(\kappa)}{2}\, \Lambda^{\ast} \left( \cdot, D^{c} \right) \right)$. 

-\textit{Conformal invariance}: if $D$ and $D^{\prime}$ are conformally equivalent domains in the plane, the pushforward of $\mu^\kappa_{D}$ via any conformal map from $D$ to $D^{\prime}$, is exactly  $\mu^\kappa_{D^{\prime}}$.

This property is a rewriting of the trivializability of the determinant line bundle on space of simple loops as postulated in \cite{Kontsevich_SLE} by Kontsevich and Suhov, while inspired by the works of Malliavin, and that is why we say that SLE loop measure is a \emph{Malliavin--Kontsevich--Suhov} measure. However, it is not known if there is a unique such measure (up to scaling) for a given central charge except in the case of $\kappa = 8/3$.

When $\kappa = 8/3$, we have $c(\kappa) = 0$. The restriction covariance becomes restriction invariance, namely,
$$ \dd \mu_{D^{\prime}}^{\kappa} (\cdot) = 1_{\cdot \subset D^{\prime} } \,\dd \mu_{D}^{\kappa}(\cdot)$$ 
using $J_\kappa \equiv 1$ and Lemma~\ref{lem:lambda_*_restriction}. 
 This measure was first studied by Werner \cite{Werner2008} prior to the more general definition \eqref{eqn.conformal.restr} for all central charge in \cite{Kontsevich_SLE}, often referred to as Werner's measure. Werner showed that there exists a unique  (up to a scaling) measure on simple loops with the conformal and restriction invariance.
 We note that the same properties hold for 
the Brownian loop measure, except that the Brownian loops are not simple. Indeed, Werner's measure can be realized as the measure on the outer boundary of the Brownian loops under $\mu^{\BL}$, which also pins down the scaling factor.

More precisely, if $V_{1}$ and $V_{2}$ are 
closed 
disjoint subsets of $\hC$, then we write
\begin{equation}\label{eqn.werner.measure.old}
\mathcal{W} \left(V_{1}, V_{2}\right):= \mu^{\BL} \left( \left\{ \text{loop } \delta \,|\, \partial_\infty \delta \cap V_{1} \neq \emptyset, \, \partial_\infty \delta \cap V_{2} \neq \emptyset \right\} \right)
\end{equation}
which is finite as shown in \cite[Lem.\,4]{NacuWerner2011}. Here, $\partial_\infty \delta$ is the \emph{outer boundary} of  Brownian loop $\delta$, namely, the boundary of the connected component of $\hC \setminus \delta$ containing $\infty$.
Intuitively, one may view $\mathcal{W}(V_{1}, V_{2})$ as another normalization of $\mu^{\BL} \left( \mathcal{L} (V_{1}, V_{2}) \right)$ since the big Brownian loops will have outer boundary near $\infty$ which does not hit $V_{1}, V_{2}$. Thanks to the invariance of Werner's measure and Brownian loop measure under M\"obius transformations, we can also write \eqref{eqn.werner.measure.old} as 
\begin{equation}\label{eqn.werner.measure}
\mathcal{W} \left(V_{1}, V_{2}\right)= \mu^{\BL} \left( \left\{ \text{loop } \delta \,|\, \partial_{z_{0}} \delta \cap V_{1} \neq \emptyset, \, \partial_{z_{0}} \delta \cap V_{2} \neq \emptyset \right\} \right),
\end{equation}
where $z_{0} \in \mathbb C$ is any point, and  $\partial_{z_{0}} \delta$ is the boundary of the connected component of $\hC \setminus \delta$ containing $z_{0}$.

\subsection{Loewner energy} \label{sec:Loewner}
In this section, we recall briefly the definition of the Loewner energy of a Jordan curve introduced in \cite{RohdeWang2021}.
 Let $\gamma: [0,1] \to \Chat = \mathbb C \cup \{\infty\}$ be a continuously parametrized Jordan curve with $\gamma (0) = \gamma(1)$. 
For every $\varepsilon>0$, $\gamma [\varepsilon, 1]$ is a chord connecting $\g(\varepsilon)$ to $\gamma (1)$ in the simply connected domain $\Chat \smallsetminus \gamma[0, \varepsilon]$.
    The \emph{Loewner energy} of $\gamma$ is 
\begin{equation}\label{def:Loewner_energy}
I^L(\gamma): = \lim_{\varepsilon \to 0} I^C_{\Chat \smallsetminus \gamma[0, \varepsilon]; \gamma(\varepsilon), \gamma(0)} (\gamma[\varepsilon, 1]) \in [0,\infty].
\end{equation}
Here, for a simple chord $\eta$ in the simply connected domain $D$ connecting $a,b \in \partial D$,
$$I^C_{D;a,b} (\eta) := \frac12\int_0^\infty(\dd W_t/\dd t)^2 \, \dd t$$ is the Dirichlet energy of the chordal driving function $W$ of $\eta$. We remark that the limit on the right-hand side of \eqref{def:Loewner_energy} is increasing as $\varepsilon \to 0_+$. Hence, the limit exists in $[0,\infty]$.
\begin{remark}
    We also have that  $I^L (\gamma) = 0$ if and only if $I^C_{\Chat \smallsetminus \gamma[0, \varepsilon]; \gamma(\varepsilon), \gamma(0)} (\gamma[\varepsilon, 1]) = 0$ for all $\varepsilon \in (0,1)$. In other words, the image of $\gamma[\varepsilon, 1]$ by a conformal map sending respectively $(\Chat \smallsetminus \gamma[0, \varepsilon]; \gamma(\varepsilon), \gamma(0))$ to $(\mathbb H;0,\infty)$ is the imaginary axis $\ii \m R_+$. It is then easy to see that this holds for all $\varepsilon$ if and only if $\gamma$ is a circle. 
\end{remark}

It is not apparent from the definition in \eqref{def:Loewner_energy}, but it is proved in \cite{W1,RohdeWang2021} that the Loewner energy does not depend on the choice of the curve's orientation nor its root. 
In fact, the parametrization-independence is further explained by the following equivalent expression of the Loewner energy. 
\begin{theorem}[See {\cite[Thm.\,1.4]{W2}}]\label{thm:intro_equiv_energy_WP}
Let $\Omega$ \textnormal(resp., $\Omega^*$\textnormal) denote the component of $\Chat \smallsetminus \g$ which does not contain $\infty$ \textnormal(resp., which contains $\infty$\textnormal) and $f$ \textnormal(resp., $g$\textnormal) be a conformal map from the unit disk $\m D = \{z \in \Chat \colon |z| < 1 \}$ onto $\Omega$ \textnormal(resp., from $\m D^* = \{z \in \Chat \colon |z| >1 \}$ onto $\Omega^*$\textnormal). We assume further that $g (\infty) = \infty$.
The Loewner energy of $\gamma$ can be expressed as
       \begin{equation} \label{eq_disk_energy}
   I^L(\gamma) = \frac{1}{\pi} \int_{\m D} \abs{\frac{f''}{f'}}^2 \dd^2 z + \frac{1}{\pi} \int_{\m D^*} \abs{\frac{g''}{g'}}^2 \dd^2 z +4 \log \abs{\frac{f'(0)}{g'(\infty)}},
 \end{equation}
 where $g'(\infty):=\lim_{z\to \infty} g'(z)$ and $\dd^2 z$ is the Euclidean area measure.
\end{theorem}
The right-hand side of \eqref{eq_disk_energy} was introduced in \cite{TT06} under the name \emph{universal Liouville action}. 
The following variation of the Loewner under a conformal map will be crucial to our proof of Theorem~\ref{thm.main}.
\begin{theorem}[See {\cite[Thm.\,4.1]{W3}}]
    If $\gamma$ is a simple loop with finite energy  and $f : A \rightarrow f(A)$ is a conformal map on an annular neighborhood $A$ of $\gamma$, then 
\begin{equation}\label{eqn.loop.ener.measure}
I^{L}(f(\gamma))- I^{L} (\gamma)= 12 \mathcal{W}\left( \gamma , A^{c} \right) - 12 \mathcal{W}( f(\gamma) , f(A)^{c} ),
\end{equation}
    where $A^c := \hC \setminus A$.
\end{theorem}

\subsection{Relation between $\Lambda^{\ast}$ and $\mathcal W$}

The main result of this section is Theorem \ref{thm.lambda.werner.compact}, where we prove a relation between the normalized Brownian loop measure $\Lambda^{\ast}$ and the corresponding quantity $\mathcal{W}$.  Corollary~\ref{cor.lambda.is.Werner}  will then be used to describe the Onsager--Machlup functional for SLE$_\kappa$ loop measures. First, we observe that although both $\Lambda^{\ast}$ and  $\mathcal{W}$ normalize Brownian loop measure, they do not equal each other. 
We add the short proof for completeness.

\begin{lemma}\label{lemma.werner.not.lambda}
    Let $K$ be a compact non-polar set and $\{K_{n}\}_{n\in \mathbb{N}}$ be a decreasing family of compact non-polar sets disjoint from $K$ such that $\cap_{n \in \mathbb N} K_n$ is a singleton $\{p\}$. We have 
\begin{align*}
    & \lim_{n\rightarrow \infty} \Lambda^{\ast} \left( K, K_{n} \right) = - \infty \quad  \text{and }\quad \lim_{n\rightarrow \infty} \mathcal{W} \left( K, K_{n} \right) = 0. 
\end{align*}
\end{lemma}

\begin{proof}
    By Lemma~\ref{lem:lambda_*_restriction} we have that 
\begin{align*}
    & \Lambda^{\ast} \left( K, K_{n} \right) - \Lambda^{\ast} \left( K, K_{n+1} \right) 
    = \mathcal{B} \left( K, K_{n} \setminus K_{n+1}; K_{n+1}^c \right) < \infty ,
\end{align*}
    and hence 
\begin{align*}
    & \Lambda^{\ast} \left( K, K_{1} \right) - \Lambda^{\ast} \left( K, K_{n} \right) = \sum_{k=1}^{n-1} \mathcal{B} \left( K, K_{k}  \setminus K_{k+1} ;  K_{k+1}^c \right) = \mathcal{B} \left( K, K_{1} \setminus K_n;  K_{n}^c \right) ,
\end{align*}
    and the latter diverges to $+\infty$ as $n \rightarrow \infty$.  Hence $\lim_{n\rightarrow \infty} \Lambda^{\ast} \left( K, K_{n} \right) = - \infty$.
 
   By monotone convergence theorem,
\begin{align*}
    & \lim_{n\rightarrow \infty }\mathcal{W} \left( K_{n}, K \right) = \mathcal{W} \left( \bigcap_{n\in \mathbb{N}} K_{n}, K \right) = \mathcal{W} \left( \{p\}, K \right) =0.
\end{align*}
The last equality holds since the probability of a two-dimensional Brownian motion hitting a singleton is zero.
\end{proof}

\begin{theorem}\label{thm.lambda.werner.compact}
    Let $A$ be an annulus, and $K\subset A$ be a connected compact subset not separating the two boundaries of $A$, and $f: A\rightarrow f(A)$ be a conformal map. Then
    \begin{align}
       &  \mathcal{W} \left( K, A^{c} \right) - \mathcal{W} ( f(K), f(A)^{c} )    =\Lambda^{\ast} \left( K, A^{c} \right) - \Lambda^{\ast} ( f(K), f(A)^{c} ) \label{eqn.lambda.is.werner.compact}
    \end{align}
\end{theorem}

\begin{remark}
    It is easy to see that the proof of Theorem \ref{thm.lambda.werner.compact} also works when $A$ is a simply connected domain. In this case, $K$ can be chosen to be any connected compact subset.
\end{remark}


\begin{proof}
    Let $z_{0} \in A \setminus K$ be fixed, and $\varepsilon$ be  small enough such that $\mathbb D_{\varepsilon} (z_{0} ) \subset A\setminus K$. A subset $I \subset A$ is a stick if $I = g( [0,1])$ for some continuous function $g: [0,1] \rightarrow A$.   Let us consider a stick $I(z_{0})$   connecting $K$ to $z_0$, and set   $I^\varepsilon(z_0) := I(z_0) \setminus \mathbb D_{\varepsilon}(z_{0})$. Moreover, let $J_{1}^{\varepsilon} (z_{0})$ and $J_{2}^{\varepsilon} (z_{0})$  be two sticks disjoint from $I^{\varepsilon} (z_{0})$ and $K$, and connecting $\partial \mathbb D_{\varepsilon}(z_{0})$ to the inner and outer boundary of $A$ respectively. This is possible since $K$ does not disconnect the two boundary components of $A$.
    
    We write $\mathcal{W} \left( K_1, K_2; D \right)$ for the total mass of Werner's measure of loops in $D$ intersecting both $K_1$ and $K_2$.  Using \eqref{eqn.werner.measure} one has  
    \begin{equation}\label{eqn.split}
    \mathcal{W} \left(K \cup  I^{\varepsilon} (z_{0}), A^{c} \cup \bigcup_{i=1}^{2} J_{i}^{\varepsilon} (z_{0})  ; \, \mathbb  D_{\varepsilon} (z_{0})^c  \right) = \mathcal{B} \left(K \cup   I^{\varepsilon} (z_{0}), A^{c} \cup \bigcup_{i=1}^{2} J_{i}^{\varepsilon} (z_{0})  ; \, \mathbb  D_{\varepsilon} (z_{0})^c  \right),
\end{equation}
    since $\mathbb  D_{\varepsilon} (z_{0})^c$ is a simply connected domain with non-polar boundary and the sets $K \cup  I^{\varepsilon} (z_{0})$ and $A^{c} \cup \bigcup_{i=1}^{2} J_{i}^{\varepsilon} (z_{0}) $  are attached to the boundary of $\mathbb  D_{\varepsilon} (z_{0})^c$. We can write \eqref{eqn.split} as 
\begin{align*}
    & \mathcal{W} \left( K \cup  I^{\varepsilon} (z_{0}), A^{c}  ; \,  \mathbb  D_{\varepsilon} (z_{0})^c  \right) + \mathcal{W} \left(K \cup  I^{\varepsilon} (z_{0}), \bigcup_{i=1}^{2} J_{i}^{\varepsilon} (z_{0}) ; \, A \setminus \mathbb D_{\varepsilon}(z_{0}) \right)
    \\
    & =  \mathcal{B} \left( K \cup  I^{\varepsilon} (z_{0}), A^{c}  ; \, \mathbb  D_{\varepsilon} (z_{0})^c  \right) + \mathcal{B} \left(K \cup   I^{\varepsilon} (z_{0}), \bigcup_{i=1}^{2} J_{i}^{\varepsilon} (z_{0}) ; \,  A \setminus  \mathbb D_{\varepsilon}(z_{0})   \right),
\end{align*}
    and similarly 
\begin{align*}
    & \mathcal{W} \left( f(K \cup I^{\varepsilon} (z_{0})), f(A)^{c}  ; \, f(\mathbb D_{\varepsilon} (z_{0}) )^c \right) \\
    & \qquad + \mathcal{W} \left(f(K \cup I^{\varepsilon} (z_{0})), \bigcup_{i=1}^{2} f(J_{i}^{\varepsilon} (z_{0}) ); \, f(A \setminus \mathbb D_{\varepsilon}(z_{0}))  \right)
    \\
    & =\mathcal{B} \left( f(K \cup  I^{\varepsilon} (z_{0})), f(A)^{c}  ; \,f(\mathbb D_{\varepsilon} (z_{0}) )^c \right) \\
    &\qquad + \mathcal{B} \left(f(K \cup I^{\varepsilon} (z_{0})), \bigcup_{i=1}^{2} f(J_{i}^{\varepsilon} (z_{0}) ); \, f(A \setminus \mathbb D_{\varepsilon}(z_{0})) \right).
\end{align*}
    By conformal invariance of both Werner's measure and the Brownian loop measure, one has 
\begin{align*}
    \mathcal{W} \left(K \cup  I^{\varepsilon} (z_{0}), \bigcup_{i=1}^{2} J_{i}^{\varepsilon} (z_{0}) ; \, A \setminus  \mathbb D_{\varepsilon} (z_{0})   \right) &=\mathcal{W} \left(f(K \cup I^{\varepsilon} (z_{0})), \bigcup_{i=1}^{2} f(J_{i}^{\varepsilon} (z_{0}) ); \, f(A \setminus \mathbb D_{\varepsilon} (z_{0}))  \right) ,
    \\
     \mathcal{B} \left(K \cup   I^{\varepsilon} (z_{0}), \bigcup_{i=1}^{2} J_{i}^{\varepsilon} (z_{0}) ; \, A \setminus  \mathbb D_{\varepsilon} (z_{0})  \right) & = \mathcal{B} \left(f(K \cup I^{\varepsilon} (z_{0})), \bigcup_{i=1}^{2} f(J_{i}^{\varepsilon} (z_{0}) ); \,  f(A \setminus \mathbb D_{\varepsilon}(z_{0}))  \right),
\end{align*}
    and hence by \eqref{eqn.split}  it follows that 
\begin{align}
    &  \mathcal{W} \left(K \cup   I^{\varepsilon} (z_{0}), A^{c}  ; \,  \mathbb  D_{\varepsilon} (z_{0})^c  \right)  - \mathcal{W} \left(f(K \cup  I^{\varepsilon} (z_{0})), f(A)^{c}  ; \, f(\mathbb D_{\varepsilon} (z_{0})  )^c\right) \label{eqn.difference}
    \\
    & =   \mathcal{B} \left(K \cup   I^{\varepsilon} (z_{0}), A^{c}  ; \,  \mathbb D_{\varepsilon} (z_{0})^c \right)  - \mathcal{B} \left(f(K \cup I^{\varepsilon} (z_{0})), f(A)^{c}  ; \, f(\mathbb D_{\varepsilon} (z_{0}))^c\right). \notag
\end{align}
  
    Now, we will get rid of the sticks $I(z_0)$. 
    We decompose 
\begin{align*}
    &  \mathcal{W} \left(K \cup  I^{\varepsilon} (z_{0}), A^{c}  ; \, \mathbb D_{\varepsilon} (z_{0})^c  \right) =  \mathcal{W} \left(K, A^{c}  ; \, \mathbb D_{\varepsilon} (z_{0})^c  \right) +   \mathcal{W} \left(    I^{\varepsilon} (z_{0}), A^{c}  ; \,  ( \mathbb D_{\varepsilon} (z_{0}) \cup K )^c  \right),
\end{align*}
    and \eqref{eqn.difference} becomes
\begin{align}
    &  \mathcal{W} \left(K, A^{c}  ; \, \mathbb D_{\varepsilon} (z_{0})^c  \right)- \mathcal{W} \left(f(K), f(A)^{c}  ; \, f(\mathbb D_{\varepsilon} (z_{0}) )^c \right)  \label{eqn.difference2}
    \\
    & \qquad  +\mathcal{W} \left(    I^{\varepsilon} (z_{0}), A^{c}  ; \, ( \mathbb D_{\varepsilon} (z_{0}) \cup K )^c  \right)  - \mathcal{W} \left(  f(I^{\varepsilon} (z_{0})), f(A)^{c}  ; \,  f(\mathbb D_{\varepsilon} (z_{0}) \cup K)^c  \right) \notag
    \\
    = &\mathcal{B} \left(K, A^{c}  ; \, \mathbb D_{\varepsilon} (z_{0})^c  \right)- \mathcal{B} \left(f(K), f(A)^{c}  ; \,  f(\mathbb D_{\varepsilon} (z_{0}) )^c\right)  \notag
    \\
    & \qquad  +\mathcal{B} \left(    I^{\varepsilon} (z_{0}), A^{c}  ; \,  ( \mathbb D_{\varepsilon} (z_{0}) \cup K )^c  \right)  - \mathcal{B} \left(    f(I^{\varepsilon} (z_{0})), f(A)^{c}  ; \,   f(\mathbb D_{\varepsilon} (z_{0}) \cup K) ^c  \right) \notag. 
\end{align}
    By monotone convergence, it follows that
\begin{align}
    & \lim_{\varepsilon \rightarrow 0}  \mathcal{W} \left(I^{\varepsilon} (z_{0}), A^{c}  ; \, ( \mathbb D_{\varepsilon} (z_{0}) \cup K )^c  \right) = \mathcal{W} \left(  I (z_{0}), A^{c}  ; \,  ( K \cup \{ z_{0} \} )^c   \right) =  \mathcal{W} \left(  I (z_{0}), A^{c}  ; \,  K ^c   \right) , \label{eqn.claim.werner}
    \\
    & \lim_{\varepsilon \rightarrow 0}  \mathcal{B} \left(  I^{\varepsilon} (z_{0}), A^{c}  ; \, ( \mathbb D_{\varepsilon} (z_{0}) \cup K )^c\right) = \mathcal{B} \left(  I (z_{0}), A^{c}  ; \,  K ^c   \right). \label{eqn.claim.loop}
\end{align}
Taking the limit as $\varepsilon \rightarrow 0$, by \eqref{eqn.claim.werner}  the left-hand-side of  \eqref{eqn.difference2} becomes
\begin{align}\label{eq:4_terms_W}
    & \mathcal{W} \left( K, A^{c} \right) - \mathcal{W} \left( f(K), f(A)^{c} \right)  +\mathcal{W} \left(  I (z_{0}), A^{c}  ; \,  K^c   \right)- \mathcal{W} \left( f(I (z_{0})), f(A)^{c}  ; \, f(K)^c  \right).
\end{align}
    The right-hand side  of \eqref{eqn.difference2}  will give terms in $\Lambda^{\ast}$. Note that $f(\mathbb D_{\varepsilon} (z_{0}))$ is comparable to  $\mathbb D_{a \varepsilon} ( f(z_{0}) )$ for  $a =  |f'(z_0)| > 0$ and $\varepsilon$ small enough since $f$ is a conformal map. We can then write the right-hand-side of \eqref{eqn.difference2} as
\begin{align*}
    & \mathcal{B} \left(K, A^{c}  ; \,  \mathbb D_{\varepsilon} (z_{0})^c  \right) - \log \log \frac{1}{\varepsilon} + \log \log \frac{1}{a \varepsilon} -\mathcal{B} \left(f(K), f(A)^{c}  ; \,  f(\mathbb D_{\varepsilon} (z_{0}) )^c \right) 
    \\
    & + \mathcal{B} \left(  I^{\varepsilon} (z_{0}), A^{c}  ; \,  ( \mathbb D_{\varepsilon} (z_{0}) \cup K )^c  \right)  - \mathcal{B} \left( f(I^{\varepsilon} (z_{0})), f(A)^{c}  ; \,  ( f(\mathbb D_{\varepsilon} (z_{0})) \cup f(K) )^c  \right) \\
    & + \log \frac{\log \varepsilon }{ \log a + \log \varepsilon},
\end{align*}
    and if we take the limit as $\varepsilon \rightarrow 0$ we get  
\begin{align}\label{eq:4_terms_Lambda}
    \Lambda^{\ast} \left( K, A^{c} \right) - \Lambda^{\ast} ( K, f(A)^{c} ) + \mathcal{B} \left(   I (z_{0}), A^{c}  ; \,  K^c   \right)- \mathcal{B} \left(  f( I (z_{0}) ), f(A)^{c}  ; \,  f(K)^c  \right),
\end{align}
    by the definition of $\Lambda^{\ast}$ and  \eqref{eqn.claim.loop}.

Finally, let $(z_n)$ be a sequence going to $K$ along $I(z_0)$. The above results replacing  $I(z_0)$ by $I(z_n)$ being the path from $z_n$ to $K$ along $I(z_0)$ still hold.  By monotone convergence, we have that 
$$\mathcal{W} \left(  I (z_{n}), A^{c}  ; \,  K^c   \right) \xrightarrow[]{n \to \infty} 0. $$
Similarly for $\mathcal{W} \left( f(I (z_{n})), f(A)^{c}  ; \, f(K)^c  \right)$,  $ \mathcal{B} \left(   I (z_{n}), A^{c}  ; \,  K^c   \right)$ and $\mathcal{B} \left(  f( I (z_{n}) ), f(A)^{c}  ; \,  f(K)^c  \right)$.
We obtain that
$$\mathcal{W} \left( K, A^{c} \right) - \mathcal{W} \left( f(K), f(A)^{c} \right)  =  \Lambda^{\ast} \left( K, A^{c} \right) - \Lambda^{\ast} ( K, f(A)^{c} )  $$
from the equality between \eqref{eq:4_terms_W} and \eqref{eq:4_terms_Lambda} which concludes the proof.
   \end{proof}

\begin{corollary}\label{cor.lambda.is.Werner}
Let $\gamma$ be a non-contractible simple loop separating the two boundaries of an annulus $A$. Then 
\begin{equation}\label{eqn.lamda.is.Werner}
\Lambda^{\ast} (\gamma , A^{c} )-\Lambda^{\ast} (f(\gamma) , f(A)^{c} ) = \mathcal{W} (\gamma ,  A^{c} )-\mathcal{W} (f(\gamma) , f(A)^{c} ).
\end{equation}
for any conformal map $f:A \rightarrow f(A)$.
\end{corollary}

\begin{proof}
We parametrize $\gamma$ continuously by $[0,1]$ such that $\gamma (0) = \gamma(1)$.
    Let $K_{n} := \gamma ([0,1-1/n])$ and $R_n := \gamma ([1-1/n, 1])$. 
    Then $K_{n}$ is a compact connected set not separating the two boundaries of $A$. Hence by \eqref{eqn.lambda.is.werner.compact}  one has that, for any $n\in \mathbb N$,
\begin{align*}
    & \mathcal{W} \left( K_{n}, A^{c} \right) - \mathcal{W} ( f(K_{n}), f(A)^{c} ) =\Lambda^{\ast} \left( K_{n}, A^{c} \right) - \Lambda^{\ast} ( f(K_{n}), f(A)^{c} )    .
\end{align*}
    From monotone convergence, we have 
\begin{align*}
 \lim_{n\rightarrow \infty}  \mathcal{W} \left( K_{n}, A^{c} \right) =  \mathcal{W} \left( \gamma, A^{c} \right) \quad \text{and } \quad \lim_{n\rightarrow \infty} \mathcal{W} ( f(K_{n}), f(A)^{c} ) =  \mathcal{W} \left( f(\gamma), f(A)^{c} \right).
\end{align*}
     For the normalized Brownian loop measure, Lemma~\ref{lem:lambda_*_restriction}  shows
\begin{align*}
    & \Lambda^{\ast} \left( \gamma, A^{c} \right) - \Lambda^{\ast} \left( K_{n}, A^{c} \right)  
    =  \mathcal{B} \left( R_{n} , A^{c} ; \, \hC  \setminus K_{n} \right).
\end{align*}
Monotone convergence and the restriction invariance of the Brownian loop measure shows
\begin{align*}
    \lim_{n\rightarrow \infty }   \left(\Lambda^{\ast} \left( \gamma, A^{c} \right) - \Lambda^{\ast} \left( K_{n}, A^{c} \right) \right) = \mathcal{B} \left( \{\gamma (1)\} , A^{c} ;\, \hC  \setminus \gamma \right)=0.
\end{align*}
Similarly, $ \lim_{n\rightarrow \infty }  \Lambda^{\ast} ( f(K_{n}), f(A)^{c} )  = \Lambda^{\ast} ( f(\gamma), f(A)^{c} )$.
This completes the proof.
\end{proof}

\subsection{Hausdorff topology on the space of simple loops}\label{sec:basis}

\begin{definition}\label{def.hausdoff.metri}
    The \textit{Hausdorff distance} $d_{h}$ of two compact sets $K_{1}, K_{2} \subset \hC$ is defined as 
    \begin{align*}
        d_{h} \left( K_{1}, K_{2} \right) := \inf_{\varepsilon \geqslant 0} \left\{ K_{1} \subset \bigcup_{x\in K_{2}} \overline{B}_{\varepsilon} (x) \text{ and } K_{2} \subset \bigcup_{x\in K_{1}} \overline{B}_{\varepsilon} (x)\right\},
    \end{align*}
    where $B_{\varepsilon} (x)$ denotes the ball of radius $\varepsilon$ around $x\in \hC$ with respect to the round metric. 
\end{definition}

The space $\mathcal{C}$ of non-empty compact subsets of $\hC$ endowed with the Hausdorff distance is a compact metric space. We then endow the space of simple loops $\mathcal{SL} \subset \mathcal{C}$ with the relative topology induced by $d_{h}$.

\begin{proposition}\label{prop.basis.topology}
Every open set $O_H$ for the Hausdorff topology on $\mathcal {SL}$ is a union of admissible neighborhoods given by \eqref{eqn.admissible.neigh}. 
\end{proposition}


\begin{proof}
    Let $\gamma \in O_{H}$, $\Omega$ the bounded connected component of $\mathbb C \setminus \gamma$ and $\Omega^c$ the unbounded connected component. 
    
    For this, let us denote by $\gamma^{1+\delta} : = f_+ (S_{1+\delta})$ and $\gamma^{1-\delta} = f_-(S_{1-\delta})$ two equipotentials on the two sides of $\gamma$, where $0 < \delta < 1/2$,  $S_r = \{ z \in \mathbb C \,|\, |z| = r\}$, $f_+$ is a conformal map $\mathbb D^c \to \Omega^c$ and $f_-$ a conformal map $\mathbb D \to \Omega$.  
    We write $Y_{\delta}$ for the doubly connected domain bounded by $\gamma^{1+\delta}$ and $\gamma^{1-\delta}$. 
    We now fix a small enough $\delta$ such that all non-contractible loops $\eta \subset Y_{\delta}$ are in $O_{H}$. By the uniformization theorem for doubly connected domains, there exists $0< r < 1$ and a conformal map $f: \mathbb A_{r} \rightarrow Y_{\delta}$. Since $f^{-1} (\gamma)$ is compact, there exists an $\varepsilon = \varepsilon_{\gamma} < 1 -r$ such that $A_{\varepsilon} := \mathbb A_{1-\varepsilon}$ contains $f^{-1} (\gamma)$. Since $A_{\varepsilon} \subset \mathbb A_{r}$, the set
\[
    O_{\gamma}:= \left\{\text{non-contractible simple loops in } f(A_{\varepsilon}) \right\}
\]
is an admissible neighborhood contained in $O_{H}$ and containing $\gamma$. We obtain that
$O_{H} = \bigcup_{\gamma\in O_{H}} O_{\gamma}$.
\end{proof}

\section{Onsager--Machlup functional for SLE loop measures}
The aim of this section is to prove Theorem \ref{thm.main}.  We begin with the following Lemma. 
Recall that $\mathcal{W} \left( K_1, K_2; D \right) \ge 0$ is the total mass of Werner's measure of loops in $D$ intersecting both $K_1$ and $K_2$, and $\mathbb D_{r} = \{ z \in \mathbb C \,|\, |z| < r\}$.
\begin{lemma}\label{lemma.continuity.Werner.1}
Let $K \subset \mathbb D$ be a compact set and $0 < \varepsilon \ll 1$ such that $K \subset \mathbb D_{1-\varepsilon}$. Then
\begin{equation}\label{eqn.continuity}
\lim_{\varepsilon \rightarrow 0} \mathcal{W} \left(K, \mathbb D \setminus  \mathbb D_{1-\varepsilon} ;  \mathbb D  \right) =0.
\end{equation}
\end{lemma}
\begin{proof}
Since Werner's measure is the measure of the outer boundary of Brownian loops,
we have that
$$ \mathcal{W} \left(K, \mathbb D \setminus  \mathbb D_{1-\varepsilon} ;  \mathbb D  \right) \le \mathcal B \left(K, \mathbb D \setminus  \mathbb D_{1-\varepsilon} ;  \mathbb D  \right).$$

  Using the decomposition of Brownian loop measure into Brownian bubble measure $m_{D,z}$ introduced in \cite{LawlerWerner2004}, see, e.g., \cite[Section 2.1.3]{LaurenceLawler2013} of Brownian bubbles in $D$ rooted at $z \in \partial D$, we have that
\begin{align*}
\mathcal B \left(K, \mathbb D \setminus  \mathbb D_{1-\varepsilon} ;  \mathbb D  \right) = \frac{1}{\pi} \int_{0}^{2\pi} \int_{1-\varepsilon}^{1} m_{\mathbb D_{r}, re^{\ii \theta}} (\{\text{bubbles intersecting } K\}) \,r dr d\theta.
\end{align*}  
Since $m_{\mathbb D_{r}, re^{\ii \theta}} (\{\text{bubbles intersecting } K\})$ is uniformly bounded for $r\in  (1-\varepsilon, 1]$ and $\theta 
\in [0,2\pi)$. The limit \eqref{eqn.continuity} follows.
\end{proof}

Fix $ 0 < r < 1$.
Recall that we write $\mathbb A_{r} : = \{ z \in \mathbb C \,|\, r < |z| < r^{-1}\}$ and $S_r := \{|z| = r\}$.

\begin{lemma}\label{lemma.continuity.Wern.measure_2}  Let $f$ be a conformal map $A : = \mathbb A_r \to \tilde A$ and $\gamma = f(S_1)$. For $\varepsilon \ll 1-r$, let $A_\varepsilon := \mathbb A_{1-\varepsilon} \subset \mathbb A_{r}$ and $\Tilde A_\varepsilon = f(A_\varepsilon)$.  Then
\begin{align*}
    \sup_{\eta \subset A_\varepsilon} \mathcal{W} \left( \eta, A^{c} \right),  \inf_{\eta \subset A_\varepsilon} \mathcal{W} \left( \eta, A^{c} \right) & \xrightarrow[]{\varepsilon \to 0\splus} \mathcal{W} \left( S_1, A^{c} \right) \\
     \sup_{\tilde \eta \subset \tilde A_\varepsilon} \mathcal{W} ( \tilde \eta, \tilde{A}^{c} ) , \inf_{\tilde \eta \subset \tilde A_\varepsilon} \mathcal{W} ( \tilde \eta, \tilde{A}^{c} ) & \xrightarrow[]{\varepsilon \to 0\splus}  \mathcal{W} ( \gamma, \tilde{A}^{c} ).
\end{align*}
Here and below, by $\eta \, \subset A_{\varepsilon}$ we always mean a non-contractible simple loop $\eta$ in $A_{\varepsilon}$.
\end{lemma}

\begin{figure}
\begin{center} 
\includegraphics[width =.4\textwidth]{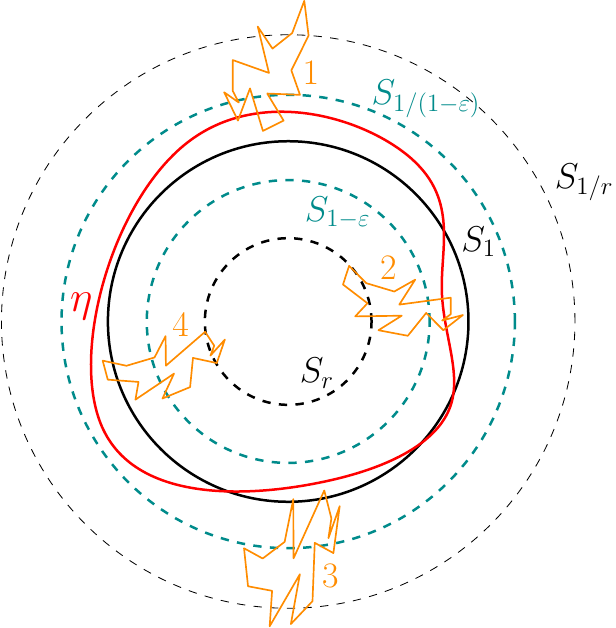}
\caption{\small The Werner's measure terms (in orange) on the right-hand-side of \eqref{eq:4_terms} numbered in order.
\label{fig_4_terms}} 
\end{center}  
\end{figure}

\begin{proof}
   To show the first limit, for $\eta \subset  A_\varepsilon$, we write $\Omega$ for the bounded connected component of $\mathbb C \setminus \eta$ and $\Omega^c$ the unbounded component. 
    We have
    \begin{align}\label{eq:4_terms}
        |\mathcal{W} \left( \eta, A^{c} \right) -  \mathcal{W} \left( S_1, A^{c} \right)| 
        \le  &  \mathcal{W} \left( \eta \cap  \mathbb D^c, S_{1/r}; \mathbb D^c \right) +  \mathcal{W} \left( \eta \cap \mathbb D,  S_{r}; \mathbb D \right) \\
          & + \mathcal{W} \left( S_1 \cap \Omega^c, S_{1/r}; \Omega^c \right) +  \mathcal{W} \left( S_1 \cap \Omega,  S_{r}; \Omega \right).  \nonumber
    \end{align}
     See Figure~\ref{fig_4_terms}.
    We can bound these terms:
    \begin{align*}
        \mathcal{W} \left( \eta \cap  \mathbb D^c, S_{1/r}; \mathbb D^c \right) &\le \mathcal{W} \left( S_{1/(1-\varepsilon)}, S_{1/r}; \mathbb D^c \right) = \mathcal{W} \left( S_{1-\varepsilon}, S_{r}; \mathbb D \right) , \\
        \mathcal{W} \left( \eta \cap \mathbb D,  S_{r}; \mathbb D \right) &\le \mathcal{W} \left( S_{1-\varepsilon}, S_{r}; \mathbb D \right),\\
         \mathcal{W} \left( S_1 \cap \Omega^c, S_{1/r}; \Omega^c \right) &\le \mathcal{W} \left( S_1, S_{1/r}; \mathbb D_{1-\varepsilon}^c \right) \le \mathcal W\left( S_{1-\varepsilon}, S_{r}; \mathbb D \right),\\
          \mathcal{W} \left( S_1 \cap \Omega,  S_{r}; \Omega \right)& \le \mathcal{W} \left( S_1,  S_{r}; \mathbb D_{1/(1-\varepsilon)} \right) \le \mathcal W \left( S_{1-\varepsilon}, S_{r}; \mathbb D \right).
    \end{align*}
We note that $ \mathcal{W} \left( S_{1-\varepsilon}, S_{r}; \mathbb D \right)  = \mathcal{W} \left( \mathbb D_r, \mathbb D \setminus \mathbb D_{1-\varepsilon}; \mathbb D \right)$.
    Using Lemma~\ref{lemma.continuity.Werner.1} and the conformal invariance of Werner's measure, we have that the four bounds above converge to $0$ as $\varepsilon \to 0$, concluding the proof of the first limit. 

The second limit follows similarly.
\end{proof}

Now, we can prove our main theorem.
    
\begin{proof}[Proof of Theorem \ref{thm.main}]
Recall that we use the notations $A:= \mathbb{A}_{r}$, $\Tilde{A} = f( \mathbb{A}_{r})$, $A_\varepsilon = \mathbb A_{1-\varepsilon}$, and $\Tilde{A}_{\varepsilon} := f (A_{\varepsilon})$,  $\gamma = f(S_1)$ and  
$$O_{\varepsilon} (\gamma ) := \left\{ \text{non-contractible simple loops in } \tilde A_{\varepsilon} \right\}.  $$
We recall that $J\left( \cdot, A^{c} \right):= \exp \left( \frac{c(\kappa)}{2}\, \Lambda^{\ast} \left( \cdot, A^{c} \right) \right)$.
By conformal restriction \eqref{eqn.conformal.restr}  and conformal invariance, we have  
\begin{align*}
& \mu^{\kappa} (O_{\varepsilon} (\gamma)) = \int 1_{\left\{ \tilde{\eta} \, \subset \tilde{A}_{\varepsilon} \right\}  } d\mu^{\kappa} ( \tilde{\eta}) = \int 1_{\left\{ \tilde{\eta} \, \subset \tilde{A}_{\varepsilon} \right\}  }  \left( \frac{d \mu^{\kappa}_{\tilde{A}} }{d \mu^{\kappa}} (\tilde{\eta} ) \right)^{-1} d \mu^{\kappa}_{\tilde{A}}( \tilde{\eta})  
\\
&=  \int    1_{\left\{ \tilde{\eta} \, \subset \tilde{A}_{\varepsilon} \right\}  }     J ( \tilde{\eta}, \tilde{A}^{c} )^{-1} d \mu^{\kappa}_{\tilde{A}}( \tilde{\eta}) = \int 1_{ \left\{ \eta \, \subset A_{\varepsilon} \right\}}    J ( f(\eta), \tilde{A}^{c} )^{-1} d \mu^{\kappa}_{A}( \eta) 
\\
& = \int 1_{ \left\{ \eta \, \subset A_{\varepsilon} \right\}}   \frac{J (\eta,A^{c})}{ J ( f(\eta), \tilde{A}^{c} )} d \mu( \eta) .
\end{align*}
By Corollary~\ref{cor.lambda.is.Werner} we have that 
\begin{align}
  \frac{J (\eta,A^{c})}{ J ( f(\eta), \tilde{A}^{c} )} &= \exp\left( \frac{c(\kappa)}{2} \, \Lambda^{\ast} \left( \eta,A^{c} \right) - \frac{c(\kappa)}{2} \, \Lambda^{\ast} ( f(\eta), \tilde{A}^{c} )  \right)  \notag
\\
&=  \exp\left( \frac{c(\kappa)}{2} \, \mathcal{W} \left( \eta, A^{c} \right) - \frac{c(\kappa)}{2}\, \mathcal{W} ( f(\eta), \tilde{A}^{c} )  \right). \label{eqn.RNderivative}
\end{align}

It follows from Lemma~\ref{lemma.continuity.Wern.measure_2} that
\begin{align*}
    \sup_{\eta \subset A_\varepsilon} \left(\mathcal{W} \left( \eta, A^{c} \right) - \mathcal{W} ( f(\eta), \tilde{A}^{c} ) \right) & \xrightarrow[]{\varepsilon \to 0\splus} \mathcal{W} \left( S_1, A^{c} \right) - \mathcal{W} ( \gamma, \tilde{A}^{c} ),\\
     \inf_{\eta \subset A_\varepsilon} \left( \mathcal{W} \left( \eta, A^{c} \right) - \mathcal{W} ( f(\eta), \tilde{A}^{c} ) \right)& \xrightarrow[]{\varepsilon \to 0\splus} \mathcal{W} \left( S_1, A^{c} \right) - \mathcal{W} ( \gamma, \tilde{A}^{c} ).
\end{align*}
By \eqref{eqn.loop.ener.measure}, the limit above equals $I^L(\gamma)/12$ since $I^L(S_1) = 0$.

Hence, we have
$$\frac{\mu^{\kappa} (O_{\varepsilon} (\gamma)) }{\mu^{\kappa} (O_{\varepsilon} (S_1)) } =  \frac{ \int 1_{ \left\{ \eta \, \subset A_{\varepsilon} \right\}}   \frac{J (\eta,A^{c})}{ J ( f(\eta), \tilde{A}^{c} )} d \mu( \eta)}{ \int 1_{ \left\{ \eta \, \subset A_{\varepsilon} \right\}}   d \mu( \eta) }  \xrightarrow[]{\varepsilon \to 0\splus} \exp \left(\frac{c(\kappa)}{24} I^L(\gamma)\right)$$
which completes the proof.
\end{proof}

\bigskip
\begin{acknowledgement}
We thank Fredrik Viklund and Pavel Wiegmann for their helpful discussions and the anonymous referees for constructive comments.  M.C. acknowledges the hospitality of IHES, where part of the work was accomplished during his visit. 
 
 This work has been supported by the NSF Grant DMS-1928930 while the authors participated in a program hosted by the Simons Laufer Mathematical Sciences Institute in Berkeley, California, during the Spring 2022 semester and by the European Union (ERC, RaConTeich, 101116694). Views and opinions expressed are however those of the authors only and do not necessarily reflect those of the European Union or the European Research Council Executive Agency. Neither the European Union nor the granting authority can be held responsible for them.
\end{acknowledgement}


\begin{thebibliography}{10}

\bibitem{Benoist_loop}
St\'{e}phane Benoist and Julien Dub\'{e}dat, \emph{An {${\rm SLE}_2$} loop
  measure}, Ann. Inst. Henri Poincar\'{e} Probab. Stat. \textbf{52} (2016),
  no.~3, 1406--1436. \MR{3531714}

\bibitem{carfagninigordina2023}
Marco Carfagnini and Maria Gordina, \emph{On the {O}nsager-{M}achlup functional
  for the {B}rownian motion on the {H}eisenberg group}, 2023.

\bibitem{ChaoDuan2019}
Ying Chao and Jinqiao Duan, \emph{The {O}nsager-{M}achlup function as
  {L}agrangian for the most probable path of a jump-diffusion process},
  Nonlinearity \textbf{32} (2019), no.~10, 3715--3741. \MR{4002397}

\bibitem{Dubedat2009partition}
Julien Dub\'{e}dat, \emph{S{LE} and the free field: partition functions and
  couplings}, J. Amer. Math. Soc. \textbf{22} (2009), no.~4, 995--1054.
  \MR{2525778}

\bibitem{LaurenceLawler2013}
Laurence~S. Field and Gregory~F. Lawler, \emph{Reversed radial {SLE} and the
  {B}rownian loop measure}, J. Stat. Phys. \textbf{150} (2013), no.~6,
  1030--1062. \MR{3038676}

\bibitem{johansson2021strong}
Kurt Johansson, \emph{Strong {S}zegő theorem on a {J}ordan curve}, Toeplitz
  operators and random matrices in memory of {H}arold {W}idom \textbf{289}
  ([2022] \copyright 2022), 427--461. \MR{4573959}

\bibitem{Kontsevich_SLE}
Maxim Kontsevich and Yuri Suhov, \emph{On {M}alliavin measures, {SLE}, and
  {CFT}}, Tr. Mat. Inst. Steklova \textbf{258} (2007), no.~Anal. i Osob. Ch. 1,
  107--153. \MR{2400527}

\bibitem{LSW_CR_chordal}
Gregory Lawler, Oded Schramm, and Wendelin Werner, \emph{Conformal restriction:
  the chordal case}, J. Amer. Math. Soc. \textbf{16} (2003), no.~4, 917--955.
  \MR{1992830}

\bibitem{LawlerWerner2004}
Gregory~F. Lawler and Wendelin Werner, \emph{The {B}rownian loop soup}, Probab.
  Theory Related Fields \textbf{128} (2004), no.~4, 565--588. \MR{2045953}

\bibitem{LeJan2006det}
Yves Le~Jan, \emph{Markov loops, determinants and gaussian fields}, arXiv
  preprint math/0612112 (2006).

\bibitem{LyonsZeitouni1999}
Terry Lyons and Ofer Zeitouni, \emph{Conditional exponential moments for
  iterated {W}iener integrals}, Ann. Probab. \textbf{27} (1999), no.~4,
  1738--1749. \MR{1742886}

\bibitem{MachlupOnsager1953b}
S.~Machlup and L.~Onsager, \emph{Fluctuations and irreversible process. {II}.
  {S}ystems with kinetic energy}, Physical Rev. (2) \textbf{91} (1953),
  1512--1515. \MR{0057766}

\bibitem{MachlupOnsager1953a}
\bysame, \emph{Fluctuations and irreversible processes}, Physical Rev. (2)
  \textbf{91} (1953), 1505--1512. \MR{0057765}

\bibitem{NacuWerner2011}
\c{S}erban Nacu and Wendelin Werner, \emph{Random soups, carpets and fractal
  dimensions}, J. Lond. Math. Soc. (2) \textbf{83} (2011), no.~3, 789--809.
  \MR{2802511}

\bibitem{RohdeWang2021}
Steffen Rohde and Yilin Wang, \emph{The {L}oewner energy of loops and
  regularity of driving functions}, Int. Math. Res. Not. IMRN (2021), no.~10,
  7715--7763. \MR{4259153}

\bibitem{schramm2000scaling}
Oded Schramm, \emph{Scaling limits of loop-erased random walks and uniform
  spanning trees}, Israel J. Math. \textbf{118} (2000), 221--288. \MR{1776084}

\bibitem{SheppZeitouni1993}
L.~A. Shepp and O.~Zeitouni, \emph{Exponential estimates for convex norms and
  some applications}, Barcelona {S}eminar on {S}tochastic {A}nalysis ({S}t.
  {F}eliu de {G}u\'{\i}xols, 1991), Progr. Probab., vol.~32, Birkh\"{a}user,
  Basel, 1993, pp.~203--215. \MR{1265050}

\bibitem{TT06}
Leon~A. Takhtajan and Lee-Peng Teo, \emph{Weil-{P}etersson metric on the
  universal {T}eichm\"{u}ller space}, Mem. Amer. Math. Soc. \textbf{183}
  (2006), no.~861, viii+119. \MR{2251887}

\bibitem{W1}
Yilin Wang, \emph{The energy of a deterministic {L}oewner chain: reversibility
  and interpretation via {${\rm SLE}_{0+}$}}, J. Eur. Math. Soc. (JEMS)
  \textbf{21} (2019), no.~7, 1915--1941. \MR{3959854}

\bibitem{W2}
\bysame, \emph{Equivalent descriptions of the {L}oewner energy}, Invent. Math.
  \textbf{218} (2019), no.~2, 573--621. \MR{4011706}

\bibitem{W3}
\bysame, \emph{A note on {L}oewner energy, conformal restriction and {W}erner's
  measure on self-avoiding loops}, Ann. Inst. Fourier (Grenoble) \textbf{71}
  (2021), no.~4, 1791--1805. \MR{4398248}

\bibitem{Wang_survey}
\bysame, \emph{Large deviations of {S}chramm-{L}oewner evolutions: a survey},
  Probab. Surv. \textbf{19} (2022), 351--403. \MR{4417203}

\bibitem{Werner2008}
Wendelin Werner, \emph{The conformally invariant measure on self-avoiding
  loops}, J. Amer. Math. Soc. \textbf{21} (2008), no.~1, 137--169. \MR{2350053}

\bibitem{Zeitouni1989}
Ofer Zeitouni, \emph{On the {O}nsager-{M}achlup functional of diffusion
  processes around non-{$C^2$}-curves}, Ann. Probab. \textbf{17} (1989), no.~3,
  1037--1054. \MR{1009443}

\bibitem{Zhan2017}
Dapeng Zhan, \emph{S{LE} loop measures}, Probab. Theory Related Fields
  \textbf{179} (2021), no.~1-2, 345--406. \MR{4221661}

\end{thebibliography}
\end{document}